\documentclass[12pt]{amsart}
\usepackage{amsmath}
\usepackage{amsfonts}
\usepackage{amsthm, upref}
\usepackage{graphicx}
 \usepackage{enumerate}
\usepackage[usenames, dvipsnames]{color}

    \numberwithin{equation}{section}

\usepackage{bm}
        
 \bmdefine\alphab{\mathbf{\alpha}}
\bmdefine\betab{\mathbf{\beta}}
\bmdefine\pib{\mathbf{\pi}}
\bmdefine\xib{\mathbf{\xi}}
\bmdefine\sigmab{\mathbf{\sigma}}

\newcommand{\comment}[1]{}
\newcommand{\eq}{\begin{equation}}
\newcommand{\en}{\end{equation}}
\newcommand{\zz}{\mathbb{Z}}
\newcommand{\rr}{\mathbb{R}}
\newcommand{\pp}{\mathbb{P}}
\newcommand{\qq}{\mathbb{Q}}
\newcommand{\nn}{\mathbb{N}}

\newcommand{\ev}{\mathbb E}

\newcommand{\ep}{\hfill $\Box$}

\begin{document}

\theoremstyle{plain}
\newtheorem{thm}{Theorem}
\newtheorem{rem}{Remark}
\newtheorem{lemma}[thm]{Lemma}
\newtheorem{prop}[thm]{Proposition}
\newtheorem{cor}[thm]{Corollary}

\theoremstyle{definition}
\newtheorem{defn}{Definition}
\newtheorem{cond}{Condition}
\newtheorem{asmp}{Assumption}
\newtheorem{notn}{Notation}
\newtheorem{prb}{Problem}

\theoremstyle{remark}
\newtheorem{rmk}{Remark}
\newtheorem{exm}{Example}
\newtheorem{clm}{Claim}

\title[Time-reversal]{Time-reversal of reflected Brownian motions in the orthant}

\author{Mykhaylo Shkolnikov and Ioannis Karatzas}
\address{University of California, Berkeley\\ Department of Statistics\\ Berkeley, CA 94720}
\email{mshkolni@gmail.com}
\address{INTECH Investment Management \\ One Palmer Square\\ Princeton, NJ 08542    and   Columbia University \\ Department of Mathematics\\ New York, NY 10027}
\email{ik@enhanced.com, ik@math.columbia.edu}

\thanks{Research partially supported by NSF grants DMS-09-05754 (I. Karatzas).}



\date{July 16, 2013}

\begin{abstract}
We determine the processes obtained from a large class of reflected Brownian motions (RBMs) in the nonnegative orthant by means of time reversal. The class of RBMs we deal with includes, but is not limited to, RBMs in the so-called \textsc{Harrison-Reiman} class \cite{HR} having diagonal covariance matrices. For such RBMs our main result resolves the long-standing open problem of determining the time reversal of RBMs beyond the skew-symmetric case treated by \textsc{R.J.$\,$Williams} in \cite{Wi}. In general, the time-reversed process itself is no longer a RBM, but its distribution is absolutely continuous with respect to a certain auxiliary RBM. In the course of the proof we introduce a novel discrete approximation scheme for the class of RBMs described above, and use it to determine the semigroups dual to the semigroups of such RBMs.
\end{abstract}

\maketitle

\section{Introduction}
\label{sec1}

Let $X(\cdot)$ be a reflected Brownian motion (RBM) in the nonnegative orthant $\rr_+^d$ with data $(\mathfrak{b},\mathfrak{A},\mathfrak{R})$, for some vector $\mathfrak{b}\in\rr^d$, a strictly positive-definite matrix $\mathfrak{A}=  (\mathfrak{a}_{i,j})_{1 \le i, j \le d}\,\in\rr^{d\times d}$, and a matrix $\mathfrak{R} = (\mathfrak{r}_{i,j})_{1 \le i, j \le d}\,\in\rr^{d\times d}$. In other words, $X(\cdot)$ is a continuous semimartingale of the form
\eq
\label{1}
X(t)=X(0)+\mathfrak{b}\,t+W(t)+\mathfrak{R}\,L(t), \quad  0 \le t < \infty 
\en  
on the canonical probability space $\,(C([0,\infty),\rr_+^d),\mathfrak{F}_\infty,\pp)\,$ endowed with the filtration $\,\mathbb{F} = \big\{\mathfrak{F}_t\big\}_{0 \le t < \infty}\,$ generated by the projection mappings. Here the process $W(\cdot)$ is a Brownian motion  with covariance matrix $\mathfrak{A}\,$, and independent of $X(0)$; and $\,L(\cdot) = \big(L_1(\cdot),\ldots,L_d(\cdot)\big)$ is the vector of local-time-like ``regulating processes" accumulated by $X(\cdot)$ on the respective faces of the boundary $\partial\rr_+^d$ of $\rr_+^d\,$. These processes are continuous, nondecreasing and adapted, and satisfy   
\eq
\label{flatoff}
L_i (0)=0 \qquad \text{and} \qquad  \int_0^\infty \mathbf{1}_{ \{ X_i(t) >0\} }\, \mathrm{d} L_i(t)=0
\en
a.s., for every $\,i=1,\ldots,d\,$. The columns $\,\overline{\mathfrak{r}}_j = ( \mathfrak{r}_{1,j}, \cdots, \mathfrak{r}_{d,j})^\prime\,$, $j=1,\ldots,d\,$ of the reflection matrix  $\,\mathfrak{R}\,$ provide the directions of reflection on the respective faces of the nonnegative orthant $\rr_+^d\,$;  here and in the sequel the symbol $'$ denotes   transposition. 

\smallskip
It was shown in \cite{TW} that a weak solution to \eqref{1} exists and is unique in distribution, if and only if the reflection matrix $\mathfrak{R}$ satisfies the so-called {\it completely-$\mathit{S}$} condition. The latter postulates that there exist  a vector $\,\lambda\in\rr_+^d\,$ such that $\,\mathfrak{R}\,\lambda>0\,$ holds componentwise, and that the same be true for all principal submatrices of $\mathfrak{R}$. We recall from Theorem 1.3 in \cite{TW} that in this case $X(\cdot)$ is a strong \textsc{Markov} process, whose semigroup $(P_t)_{t\ge0}$ is \textsc{Feller} continuous. The completely-$\mathit{S}$ condition is satisfied, in particular, by reflection matrices of the \textsc{Harrison-Reiman} \cite{HR} type $\, \mathfrak{R}= \mathfrak{I}_d - \mathfrak{Q}\,$, where $\, \mathfrak{I}_d\,$ is the identity matrix and $\,\mathfrak{Q}\,$ is a $\,(d\times d)\,$ matrix with zeros on its diagonal and spectral radius strictly less than 1; in this case the resulting RBM is given by the image of a deterministic map applied to the driving Brownian motion, and the solution of (\ref{1}) is thus strong and pathwise unique. 

\subsection{Time-Reversal} \label{sec1.0}

In contrast to the case of a Brownian motion and, more generally, of diffusion processes, very little is known about the behavior of RBMs in the orthant under time reversal. As in the case of diffusion processes (cf. \cite{HP}, \cite{M}), one expects the appearance of a logarithmic gradient of the marginal density of $X(\cdot)$ in the drift of the time-reversed process, and part of the difficulty in the analysis is the lack of knowledge about the regularity of such densities when one approaches the boundary of the orthant (the usual regularity results for partial differential equations with oblique derivative boundary conditions do not apply, due to the non-smoothness of the domain). A notable exception is the so-called {\it skew-symmetric case} studied by \textsc{Harrison} \& \textsc{Williams} \cite{HW2} and \textsc{Williams} \cite{Wi3}, in which a certain compatibility relation between $\,\mathfrak{A}\,$ and $\,\mathfrak{R}\,$ (see \eqref{sksy} below) guarantees that the invariant distribution of the RBM is a product of exponential distributions. In this case, the time-reversal of the stationary RBM was determined in \cite{Wi} and was shown to be given by yet another RBM. This is not surprising, as the two RBMs have been known to be in duality with each other (cf. \cite{Wi3}). In comparison with the diffusion case, the main novelty here is the appearance of the dual reflection matrix $\,\mathfrak{R}^*= (\mathfrak{r}_{i,j}^*)_{1 \le i, j \le d}\,$, obtained from $\mathfrak{R}$ by reflecting its columns with respect to the respective inward unit normal vectors on the faces of $\rr_+^d$.  

\smallskip
The results of \cite{Wi} and \cite{Wi3} have raised the questions of determining the time reversal and the dual process for RBMs beyond the skew-symmetric case. Our Theorems \ref{main_thm} and \ref{dual} resolve these long-standing open problems under the following assumption. 

\begin{asmp}\label{main_asmp}
The matrix $\,\mathfrak{A}\,$ satisfies $\,\mathfrak{a}_{i,i}>\sum_{j\neq i} |\mathfrak{a}_{i,j}|$, $j=1,\ldots,d\,$; the matrix $\mathfrak{R}$ is invertible; $\mathfrak{R}^{-1}\,\mathfrak{b}<0$ holds componentwise; and the row sums of the matrices $\mathfrak{R}\,$, $\mathfrak{R}^*$ and of all their principal submatrices are positive. 
\end{asmp}

\smallskip

With $\,\mathfrak{A}\,$ strictly positive-definite and $\,\mathfrak{R}\,$ invertible, the condition $\mathfrak{R}^{-1}\,\mathfrak{b}<0$ is necessary and sufficient for the existence of an invariant probability measure ${\bm \nu}$ for $X$ (see Theorem (2) in \cite{HW}), which is then known to have an infinitely differentiable density function $\,p:\;\rr^d_{++} \rightarrow(0,\infty)\,$ on the interior $\rr^d_{++}$ of $\rr^d_+$ (by \textsc{Weyl}'s Lemma; see \cite{Wi2} for more details). The condition on the row sums of $\mathfrak{R}\,$, $\mathfrak{R}^*$ and all their principal submatrices is somewhat stronger than the completely-$\mathit{S}$ condition, in that it imposes that the vector $\lambda$ in the definition of the completely-$\mathit{S}$ condition can be chosen as the vector $(1,\ldots,1)'\in\rr_+^d$,  and that the corresponding restriction be true for all principal submatrices as well. Note, however, that we do not impose any relation  between the matrices $\,\mathfrak{A}\,$ and $\,\mathfrak{R}\,$, so that $X(\cdot)$ need not   be skew-symmetric.  

\smallskip
The main result of our paper determines the time reversal of stationary RBMs under Assumption \ref{main_asmp}. In particular,  we note that every RBM in the \textsc{Harrison}-\textsc{Reiman} class of \cite{HR} with a diagonal covariance matrix can be turned into a RBM as in Assumption \ref{main_asmp} by a suitable rescaling of the coordinates;  see the first paragraph in the proof of Theorem 1 in \cite{HR}. Thus   we are able to determine the time-reversal of stationary RBMs in the \textsc{Harrison}-\textsc{Reiman} class. It is also easy to see that Assumption \ref{main_asmp} allows  for RBMs which are not given by a rescaling of coordinates of a RBM in the \textsc{Harrison}-\textsc{Reiman} class. For example, one can choose $\mathfrak{R}$ as a $2\times2$ matrix, whose bottom-left entry is negative whereas all other entries are positive. 

\subsection{The Main Result} \label{sec1.1}

To set the stage for our main result, we consider the auxilliary RBM $\,\tilde{X} (\cdot)\,$ given by
\eq\label{2}
\tilde{X}(t)=\tilde{X}(0)-\mathfrak{b}\,t + B(t) + \mathfrak{R}^*\,\tilde{L}(t),\quad 0 \le t < \infty
\en
on a copy $\,\big(C([0,\infty),\rr_+^d), \mathfrak{F}_\infty,\tilde{\pp}\big)$, $\,\mathbb{F} = \big\{\mathfrak{ F}_t\big\}_{0 \le t < \infty}\,$ of the canonical filtered probability space above. Here $B(\cdot)$ is a Brownian motion with the same distribution as $W(\cdot)$, the matrix $\mathfrak{R}^*$ is the dual reflection matrix defined in the previous subsection, and $\tilde{L}(\cdot)=\big(\tilde{L}_1(\cdot),\ldots,\tilde{L}_d (\cdot)\big)'$ is the vector of local time processes accumulated by $\tilde{X}(\cdot)$ on the faces of $\rr_+^d$. We fix a time-horizon $T\in(0,\infty)$, and introduce a new measure $\mathbb{Q}\,$ on $\,\mathfrak{F} (T)\,$ by 
\begin{equation}
\label{Q}
\mathbb{Q} (A)\,:=\ev^{\widetilde{\pp}}\bigg[\exp\Big(-2\,\sum_{j=1}^d \,\frac{\mathfrak{b}_j\,\mathfrak{r}_{j,j}}{\mathfrak{a}_{j,j}}\,\tilde{L}_j (T)\Big)\,\,
\frac{\,p(\tilde{X}(T))\,}{p(\tilde{X}(0))}\,\,\mathbf{1}_A \,\bigg].
\end{equation}
As we show below, $\,\mathbb{Q}\,$ is in fact a probability measure; this is a consequence of an appropriate duality relation between the processes $\,\tilde{X}(\cdot)\,$ and $\,X(\cdot)\,$  (see the paragraph following (\ref{mainiden})). 

\medskip

Throughout the paper  we shall work under Assumption \ref{main_asmp}, and let the random variables $\,X(0)\,$ of \eqref{1} and $\,\tilde{X}(0)\,$ of \eqref{2} be distributed according to $\,{\bm \nu}\,$, so that $\,X(\cdot)\,$ is a stationary process. We are now ready to state our main result.  
 
\begin{thm}
\label{main_thm}
The measure $\mathbb{Q}$ of \eqref{Q} is a probability measure; under it, the 
distribution of $\,\tilde{X}(t),\,0\le t\leq T$   is the same as the distribution of $\,X(T-t), \, 0 \le t \le T\,$ under the original measure $\mathbb{P}\,$. 

Moreover, if the probability density function  $\,p\,$ is twice continuously differentiable and strictly positive throughout $\,\rr_+^d\,$, then the process 
\eq\label{expmart}
Z(t) := 
\exp\Big(\int_0^t \big(\nabla\log p)(\tilde{X}(s)\big)' \mathrm{d}B(s)-\frac{1}{2}\int_0^t \big((\nabla\log p)'\, \mathfrak{A}\,(\nabla\log p)\big)(\tilde{X}(s))\,\mathrm{d}s\Big),\;\;\; 0 \le t \le T
\en
is well-defined and a $\,\tilde{\pp}$-martingale; whereas $\,\tilde{X}(t),\, 0 \le t \le T\,$ is a continuous $\,\mathbb{Q}$-semimartingale   with  decomposition
\eq\label{semimart}
\tilde{X}(t)=\tilde{X}(0)-\mathfrak{b}\,t+\int_0^t \mathfrak{A}^{1/2}\,(\nabla\log p)(Y(s))\,\mathrm{d}s+{\bm\beta}(t)+\mathfrak{R}^*\,\tilde{L}(t)\,,\qquad 0 \le t \le T  
\en
and  $\,{\bm \beta} (\cdot)$   a $\mathbb{Q}-$Brownian motion with diffusion matrix $\,\mathfrak{A}\,$.  
\end{thm}  

\begin{rem} 
{\rm 
When the covariance matrix $\mathfrak{A}$ and the reflection matrix $\mathfrak{R} $ satisfy the skew-symmetry condition
\eq\label{sksy}
2\, \mathfrak{A}\,=\, \mathfrak{R}\,\mathfrak{D} + \mathfrak{D}\,\mathfrak{R}'\,,
\en
where $\,\mathfrak{D}\,$ is the diagonal matrix with the same diagonal entries as $\,\mathfrak{A}\,$, it was shown by \textsc{Harrison \& Williams} \cite{HW2} that the invariant probability density function is given by a product of exponentials:
\eq
p(x)\,=\,C\,\exp \big( - \eta' \,  x \big) \,, \quad x \in \rr_+^d\,, \qquad \text{with} \quad \eta \,=\, -2 \,\big( \mathfrak{R}\, \mathfrak{D} \big)^{-1} \,\mathfrak{b}\,. 
\en
In this case the dynamics of \eqref{semimart} for the time-reversal of $\,X(\cdot)\,$ are again those of a RBM in the orthant with drift $\,-\mathfrak{b}-2\,\mathfrak{A}^{1/2}\big(\mathfrak{R}\,\mathfrak{D})^{-1}\,\mathfrak{b}\,$, the same covariance matrix $\,\mathfrak{A}\,$, and reflection matrix $\,\mathfrak{R}^*\,$. This recovers the result of \cite{Wi}.
}
\end{rem}

\subsection{Preview} 
\label{sec1.3}

The rest of the paper is devoted to the proof of Theorem \ref{main_thm}. This  proof is broken down into three main steps, which are carried out in sections \ref{sec2}, \ref{sec3} and \ref{sec4}, respectively. The main ideas of the proof are the following. 

First, we introduce continuous-time \textsc{Markov} chains with discrete state-spaces approximating the RBM $\,X(\cdot)\,$ of Theorem \ref{main_thm}, and prove that they converge -- together with several important observables -- to their continuum analogues (section \ref{sec2}). In section \ref{sec3} we determine the dual processes of the approximating \textsc{Markov} chains and, passing to a suitable scaling limit, obtain an appropriate duality relation between $\,X(\cdot)\,$ and $\,\tilde{X}(\cdot)\,$. The convergence result (Theorem \ref{conv}) and the duality result (Theorem \ref{dual}) are of interest in their own right. Lastly, in section \ref{sec4} we use the duality relation to show that the finite-dimensional distributions of the time-reversal of $\,X(t),\,0\le t \le T\,$ under the original measure $\mathbb{P}\,$, and of the process $\,\tilde{X}(t), \, 0 \le t \le T\,$ under the new measure $\,\qq\,$, are the same.

\smallskip
One of the main ingredients in the proof of the main result is a construction of continuous-time  \textsc{Markov} chains with discrete state-spaces, approximating a RBM as in Assumption \ref{main_asmp}. This  is of interest in its own right and, to the best of our knowledge,   the first such construction which is able to handle RBMs beyond the \textsc{Harrison-Reiman} class. For different kinds of approximations of RBMs in the \textsc{Harrison-Reiman} class we refer the reader to \cite{R}, \cite{BW} and \cite{KPS}. The reason we have to work under Assumption \ref{main_asmp}, rather than allowing $\,\mathfrak{A}\,$ to be an arbitrary symmetric positive definite matrix and
$\,\mathfrak{R}\,$ to be completely-$\mathit{S}$, is that we were not able to construct a sequence of approximating \textsc{Markov} chains in this generality. 

In every other respect  our argument is completely general, and should be able to handle the time-reversal of a generic RBM in a generic domain  once the appropriate approximation theory and the weak uniqueness of the limiting RBM have been established. A particularly interesting such case would be that of a RBM in the orthant with   degenerate covariance matrix $\,\mathfrak{A}\,$ satisfying $\,\mathfrak{a}_{j,j}>0$, $j=1,\ldots,d\,$; in this setting it is not even clear under which conditions on $\,\mathfrak{b}\,$ and $\,\mathfrak{R}\,$ the change of measure \eqref{Q} can be made sense of. 

\section{Discrete approximation processes} \label{sec2}

The starting point of our approach is the definition of the discrete approximation processes $\,X^n(\cdot)$, $n\in\nn\,$ announced in the introduction. 

To prepare the ground for the construction of the approximating chains, we first let $\,S^n (\cdot)$, $n\in\nn\,$ be a sequence of continuous-time \textsc{Markov} chains on the lattices $\,\frac{1}{\sqrt{n}}\,\zz^d$, $n\in\nn\,$ whose jump rates are homogeneous in space and time. The jumps of these \textsc{Markov}  chains are of the forms 
$$\pm e_i/\sqrt{n} \qquad \text{ and} \qquad \pm e_i/\sqrt{n}\,\pm e_j/\sqrt{n}\,,$$ where $\,e_1,\ldots,e_d\,$ denote the standard basis vectors in $\,\rr^d\,$. The corresponding jump rates are defined as
\begin{eqnarray}
&&n\,\mathfrak{a}_{i,i}/\zeta_i:=n\,\Big(\mathfrak{a}_{i,i}-\sum_{j\neq i} |\mathfrak{a}_{i,j}|\Big)/2\;\;\;\mathrm{for\;jumps\;by}\;\;\;e_i/\sqrt{n}\,, \label{rate1}\\
&&n\,\mathfrak{a}_{i,i}/\zeta_i-\sqrt{n}\,\mathfrak{b}_i\;\;\;\mathrm{for\;jumps\;by}\;\;\; -e_i/\sqrt{n}\,, \label{rate2}\\
&&n\,(\mathfrak{a}_{i,j})_+/2\;\;\;\mathrm{for\;jumps\;by}\;\;\;\pm(e_i/\sqrt{n}+e_j/\sqrt{n})\,, \label{rate3}\\
&&n\,(\mathfrak{a}_{i,j})_-/2\;\;\;\mathrm{for\;jumps\;by}\;\;\;\pm(e_i/\sqrt{n}-e_j/\sqrt{n})\,, \label{rate4}
\end{eqnarray}
where the constants $\,\zeta_1,\ldots,\zeta_d>0\,$ are well-defined by \eqref{rate1} due to Assumption \ref{main_asmp}. Here $\,(\cdot)_+$ and $(\cdot)_-\,$ stand for the positive and negative parts, respectively. As $\,n\to\infty\,$, the \textsc{Markov} chains $\,S^n (\cdot)$, $n\in\nn\,$ converge in distribution  to $\,(t\mapsto\mathfrak{b}\,t+W(t))\,$, with respect to the \textsc{Skorokhod} topology on $\,D([0,\infty,\rr^d))\,$; this is a consequence of \textsc{Donsker}'s Invariance Principle, and of the Lemma on page 151 in \cite{Bi}.

\subsection{The Approximating \textsc{Markov} Chains.} 
\label{sec2.0}

Next, we define the continuous time Markov chains $\,X^n(\cdot)$, $n\in\nn\,$ on $\,\frac{1}{\sqrt{n}}\,\zz_+^d$, $n\in\nn\,$, respectively, by the following procedure. When $\,X^n(\cdot)\,$ is at a site with all coordinates positive, we let its jump rates be given by the corresponding jump rates of $\,S^n(\cdot)\,$. Now, suppose $\,X^n(\cdot)\,$ has reached a site $\,x\,$ on a boundary face
\[
\mathcal{B}_I:=\Big\{x\in\frac{1}{\sqrt{n}}\,\zz_+^d:\;x_i=0,\;i\in I,\quad x_j>0,\;j\notin I\;\Big\}
\]
for some non-empty set $\,I\subset\{1,\ldots,d\}\,$. 

\smallskip
\noindent
$\bullet\;$ If $\,I=\{i\}\,$ for some $\,i\in\{1,\ldots,d\}\,$, then we allow $\,X^n(\cdot)\,$ to jump from $\,x\,$ to sites of the form $$\,x+e_i/\sqrt{n}+\mathrm{sgn}(\mathfrak{a}_{i,j})e_j/\sqrt{n}\,,  \qquad \,x+e_i/\sqrt{n}\,,\qquad \,x\pm e_j/\sqrt{n}  
$$ 

\medskip
\noindent
for some $j\in\{1,\ldots,d\}\backslash\{i\}\,$, with corresponding jump rates given respectively by
\eq\label{boundaryrates1}
n\,\mathfrak{r}_{i,i}\,\frac{\,|\mathfrak{a}_{i,j}|\,}{\mathfrak{a}_{i,i}}\,,
\qquad n\,\mathfrak{r}_{i,i}\,\frac{2}{\,\zeta_i\,}
\qquad\mathrm{and}\qquad n\,(\mathfrak{r}_{j,i})_{\pm}+n\,c_{\{i\},j,\pm}\,\,.
\en
Here $\,\mathrm{sgn} =\mathbf{1}_{[0,\infty)}-\mathbf{1}_{(-\infty,0)}\,$ is the sign function. The constants $\,c_{\{i\},j,\pm}$ are assumed to be positive,  and to satisfy 
 $$
 \,c_{\{i\},j,+}-c_{\{i\},j,-}\,=\,-\mathfrak{r}_{i,i}\,\frac{\,\mathfrak{a}_{i,j}\,}{\mathfrak{a}_{i,i}}\,; 
$$ 
 they will be specified concretely later. We remark that $\,\mathfrak{r}_{i,i}\,$ is positive by Assumption \ref{main_asmp}.

\medskip
\noindent
$\bullet\;$ If $\,|I|\geq2\,$, then we allow $\,X^n(\cdot)\,$ to jump from $\,x\,$ to sites of the following four types:
\eq
\begin{split}
& x+e_i/\sqrt{n} \quad\mathrm{if\;\;\;}i\in I,\qquad x+e_j/\sqrt{n} \quad\mathrm{if\;\;\;}j\notin I, \\
& x-e_j/\sqrt{n} \quad\mathrm{if\;\;\;}j\notin I,\qquad x+e_i/\sqrt{n}+e_j/\sqrt{n} \quad\mathrm{if\;\;\;}i,j\in I,\,|I|=2,\,\mathfrak{a}_{i,j}>0\,.
\end{split}
\en
We define the corresponding jump rates as
\begin{eqnarray}\label{boundaryrates}
\quad n\,c_{I,i}\,\sum_{\ell\in I} \mathfrak{r}_{i,\ell}\,,
\;\;\; n\,\sum_{\ell\in I} (\mathfrak{r}_{j,\ell})_+ + n\,c_{I,j,+}\,,
\;\;\; n\,\sum_{\ell\in I} (\mathfrak{r}_{j,\ell})_- + n\,c_{I,j,-}\,, 
\;\;\; n\,c_{I,i,j}\,\sum_{\ell\in I} \mathfrak{r}_{i,\ell}\,,
\end{eqnarray}
respectively, where the constants are positive and obey $$\,c_{I,i}=1\, ~\,\,~ \text{if}~\,|I|\neq2\, , \qquad \,c_{\{i,j\},i}+c_{\{i,j\},i,j}=1\,\qquad \text{ and}\quad\,c_{I,j,+}=c_{I,j,-}\,,$$ and will be specified concretely later. We note that   Assumption \ref{main_asmp} implies  $\sum_{\ell\in I} \mathfrak{r}_{i,\ell}>0\,$. 

\subsection{Semimartingale Decompositions.} 
\label{sec2.1}

Next, we fix an $\,n\in\nn\,$ and derive the semimartingale decomposition of $\,X^n(\cdot)\,$. To this end, we introduce the clocks 
\eq
\label{clocks}
T^n(t)=\int_0^t \mathbf{1}_{\{X^n_1(s)>0,\ldots,X^n_d(s)>0\}}\,\mathrm{d}s\qquad\text{and}\qquad L^n_i(t)=\sqrt{n}\cdot\int_0^t \mathbf{1}_{\{X^n_i(s)=0\}}\,\mathrm{d}s 
\en
for $\,i=1,\ldots,d  \,$. Letting $\,M^n(\cdot)\,$ be the martingale given by the compensated sum of jumps of $\,X^n(\cdot)\,$, one can decompose $\,X^n(\cdot)\,$ according to 
\eq\label{decompn}
X^n(t)=X^n(0)+C^n(t)+M^n(t)+{\mathfrak R}\,L^n(t),\qquad 0\leq t<\infty,
\en
where $\,C^n(\cdot)=\mathfrak{b}\,T^n(\cdot)\,$ is the compensator of the jumps of $\,X^n(\cdot)\,$ originating from sites with positive coordinates and $\,L^n(\cdot)=(L^n_1(\cdot),\ldots,L^n_d(\cdot))'\,$. 

\medskip

We have the following convergence result.

\begin{thm}\label{conv}
Suppose that the initial positions $\,X^n(0)$, $n\in\nn\,$ converge in distribution to a limit $\,X(0)\,$. Then the 
processes $\,X^n(\cdot)$, $n\in\nn\,$ converge in distribution with respect to the \textsc{Skorokhod} topology on $\,D([0,\infty),\rr^d)\,$ to the RBM $\,X(\cdot)\,$ of Theorem \ref{main_thm} with   initial condition $\, X(0)\,$. Moreover, one has the convergences in distribution
\eq\label{convergences}
C^n(\cdot)\Longrightarrow (t\mapsto\mathfrak{b}t),\qquad M^n(\cdot)\Longrightarrow W(\cdot),\qquad L^n(\cdot)\Longrightarrow L(\cdot)\,,
\en
as well as  
\eq\label{bdryprop}
\Big(t\mapsto\sqrt{n}\cdot\int_0^t \mathbf{1}_{\{X^n_i(s)=X^n_j(s)=0\}}\,\mathrm{d}s\Big) \Longrightarrow 0 
\en
for any $\,1\leq i<j\leq d\,$ and with respect to the same topology.
\end{thm}

\begin{proof}
To deduce \eqref{convergences}, one can follow the lines of the proof of Theorem 8 in \cite{KPS}. For the sake of completeness, we describe the main steps. 

First, one establishes the tightness of $\,C^n(\cdot)$, $n\in\nn\,$ and $\,M^n(\cdot)$, $n\in\nn\,$ on $\,D([0,\infty),\rr_+^d)$ by verifying the criterion in Corollary 3.7.4 of \cite{EK}; one  uses the facts that the sequence $\,S^n(\cdot)$, $n\in\nn\,$ converges, and that $M^n(\cdot)\,$ can be viewed as a \textsc{Lipschitz}-continuous time change of the compensated version of $\,S^n(\cdot)\,$ with the \textsc{Lipschitz} constant being bounded uniformly in $n$. Next, one employs the decomposition \eqref{decompn} together with the oscillation inequalities in Theorem 5.1 of \cite{Wi4} to deduce the tightness of the processes $\,(X^n(\cdot),C^n(\cdot),M^n(\cdot),L^n(\cdot))$, $n\in\nn\,$ on $\,D([0,\infty),\rr^{4d})\,$ from the tightness of the processes $\,(C^n(\cdot),M^n(\cdot))$, $n\in\nn\,$ via the criterion of Corollary 3.7.4 in \cite{EK}. At this point, the tightness of the sequence $\,L^n(\cdot)$, $n\in\nn\,$ shows that $\,T^n(\cdot)\Rightarrow(t\mapsto t)\,$ in $\,D([0,\infty),\rr)\,$, and consequently every limit point of $\,C^n(\cdot)+M^n(\cdot)$, $n\in\nn\,$ must have the same distribution as $\,(t\mapsto\mathfrak{b}\,t+W(t))\,$. Finally, one can show that every limit point of $\,(X^n(\cdot),C^n(\cdot),M^n(\cdot),L^n(\cdot))$, $n\in\nn\,$ satisfies \eqref{1} and \eqref{flatoff} and \eqref{convergences} readily follows from the uniqueness in distribution of the RBM $\,X (\cdot)\,$.      

\medskip

It remains to show \eqref{bdryprop}. It is clear from \eqref{convergences} that the sequence of prelimit expressions in \eqref{bdryprop} is tight and that every limit point $\,\Lambda (\cdot)\,$ in \eqref{bdryprop} must satisfy
\[
\forall\,0\le t_1<t_2<\infty:\quad \Lambda(t_2)-\Lambda(t_1)\leq L_i(t_2)-L_i(t_1).  
\]
Now, the same argument as in the proof of Proposition 9 in \cite{KPS} and the above  observation yield
\[
\forall\,0\le t_1<t_2<\infty:\quad \Lambda(t_2)-\Lambda(t_1)=\int_{t_1}^{t_2} \mathbf{1}_{\{X_j(s)=0\}}\,\mathrm{d}\Lambda(s)
\leq \int_{t_1}^{t_2} \mathbf{1}_{\{X_j(s)=0\}}\,\mathrm{d}L_i(s).
\]
However, the latter expression vanishes due to the boundary property of RBMs established in \cite{RW} and we end up with the claim $\, \Lambda (\cdot) \equiv 0\,$ of \eqref{bdryprop}.
\end{proof}

\subsection{Truncated Markov Chains.} 
\label{sec2.2}

It will be convenient for us to work with \textit{truncated} versions $\,Y^n(\cdot)$, $n\in\nn\,$ of the \textsc{Markov} chains $\,X^n(\cdot)$, $n\in\nn\,$ with state spaces of the form $$\mathcal{Y}^n:=\Big(\frac{1}{\,\sqrt{n\,}\,}\,\zz_+^d\Big)\cap\big[0,K_n\big]^d\,,$$ whose jump rates are given by the restriction of the jump rates of $\,X^n(\cdot)$, $n\in\nn\,$ to sites in $\,\mathcal{Y}^n\,$. Clearly, for any fixed $\,T\in(0,\infty)\,$ one can let the sequence $\,(K_n)_{n\ge1}\,$ grow fast enough to ensure that $\,Y^n(\cdot)$, $n\in\nn\,$ admit decompositions of the form \eqref{decompn}, for which the convergences \eqref{convergences} and \eqref{bdryprop} hold on $\,D([0,T],\rr^d)\,$. 

\smallskip
We write $ q^n_{x,y} $ for the rate at which $ Y^n(\cdot) $ jumps from site $\,x\,$ to site $\,y\,$. We  shall denote by   $\,\partial\mathcal{Y}^n\,$    the collection of sites in $\,\mathcal{Y}^n\,$ with at least one coordinate  equal to zero,   and   by $\, \partial_1\mathcal{Y}^n $ the collection of sites in $ \mathcal{Y}^n $   with at least one coordinate less than or equal to $\,\frac{1}{\sqrt{n}}\,$.   

\section{Duality} \label{sec3}

In this section we establish a duality relation between the processes $X(\cdot)$ and $\tilde{X}(\cdot)$, which will allow us to deduce Theorem \ref{main_thm}. We start with its discrete version. 

\begin{lemma} \label{duality_disc} 
Fix an $ n\in\nn $, and let $\,Y^n(\cdot)\,$ be the    \textsc{Markov} chain with   state space $\,\mathcal{Y}^n\,$ and generating matrix $q^n$ defined in subsection \ref{sec2.2}. Consider the  \textsc{Markov} semigroup $\,\big(\hat{P}^n_t\big)_{t\ge0}\,$ corresponding to the generating matrix $\,\hat{q}^{\,n}:=(q^n)'\,$, the transpose of $q^n$. 

There exists then a continuous-time \textsc{Markov} chain $\tilde{Y}^n(\cdot)$ on $\,\mathcal{Y}^n\,$, such that
\begin{enumerate}[(i)]
\item a decomposition
\eq\label{dual_decomp}
\tilde{Y}^n(t)=\tilde{Y}^n(0)+\tilde{C}^n(t)+\tilde{M}^n(t)+\mathfrak{R}^*\,\tilde{L}^n(t),\qquad 0\leq t<\infty\,
\en
analogous to \eqref{decompn} holds, so that, in particular, 
\eq
\tilde{L}^n_i(t)=\sqrt{n\,}\cdot\int_0^t \mathbf{1}_{\{\tilde{Y}^n_i(s)=0\}}\,\mathrm{d}s,\qquad i=1,\ldots,d\,;
\en
\item for each $\,t\in[0,\infty)\,$ and $\,x\in\mathcal{Y}^n\,$, the measure $\delta_x\hat{P}^n_t$ is absolutely continuous with respect to the distribution of $\,\tilde{Y}^n(t)\,$ given $\,\tilde{Y}^n(0)=x\,$, with   density of the form
\eq
\label{com}
\exp\Big(\int_0^t V^n(\tilde{Y}^n(s))\,\mathrm{d}s\Big)
\prod_{x\in\partial_1\mathcal{Y}^n} \prod_{y\sim x} \Big(\frac{\hat{q}^n_{x,y}}{\tilde{q}^n_{x,y}}\Big)^{\tilde{N}^n_{x,y}(t)}
\,\exp\Big(-\sum_{x\in\partial_1\mathcal{Y}^n} \tilde{T}^n_x(t)\,\sum_{y\sim x} (\hat{q}^n_{x,y}-\tilde{q}^n_{x,y})\Big).
\en
\end{enumerate} 

Moreover, for any functions $f,\,g\,$ in $\, C_c(\rr_+^d)$, the space of continuous functions on $\rr_+^d$ with compact support contained in the interior $\,\rr_{++}^d\,$of $\,\rr_+^d\,$, and for every $\, t \in [0,\infty)\,$, we have 
\eq
\label{duality0}
\sum_{x\in\mathcal{Y}^n} f(x)\,\big(\hat{P}^n_t\,g\big)(x)
= \sum_{x\in\mathcal{Y}^n} \ev^x\big[f(Y^n(t))\big]\,g(x)\,.
\en
\end{lemma}

\medskip
\noindent
\textsc{Remark on Notation:} 
In the expressions of (\ref{com}), (\ref{duality0}) we have set 
\eq\label{potential}
V^n(x)\,:=\,\sum_{y\sim x}\, \big(q^n_{y,x}-q^n_{x,y}\big)\,,\qquad x\in\mathcal{Y}^n\,;
\en
the notation $\,y\sim x\,$ means that site $\,y\,$ can be reached by $\,\tilde{Y}^n(\cdot)\,$ from $\,x\,$ in one jump; the number $\,\tilde{q}^{\,n}_{x,y}\,$ stands for the rate of the jump from $\,x\,$ to $\,y\,$ by $\,\tilde{Y}^n(\cdot)\,$; the quantity $\,\tilde{T}^n_x(t) = \int_0^t \mathbf{ 1}_{\{ \tilde{Y}^n (s)=x \}} \mathrm{d} s\,$ is the time spent by $\,\tilde{Y}^n(\cdot)\,$ at $\,x\,$ during the time interval $\,[0,t]\,$; whereas the quantity $\,\tilde{N}^n_{x,y}(t)\,$ denotes the number of jumps from $ x $ to $ y $ by $\,\tilde{Y}^n(\cdot)\,$ up to time $\,t\,$. In the same spirit, and for later usage in (\ref{simplified2}), we define the quantity $\,\tilde{N}^n_{I,J}(t)\,$ as the number of jumps from 
the boundary $\,\mathcal{B}_I\,$ to the boundary $\,\mathcal{B}_J\,$, for two distinct subsets $\, I, \, J$ of $\, \{ 1, \cdots, d\}\,$. 
 
\begin{proof}
We start by considering the matrix indexed by $\,\mathcal{Y}^n\,$, whose off-diagonal elements coincide with those of $\,\widehat{q}^{\,\,n} = ( q^n)'\,$ and whose rows sum to zero. Such a matrix is the generator of a continuous-time \textsc{Markov} chain $\,\hat{Y}^n(\cdot)\,$ on $\,\mathcal{Y}^n\,$ with jump rates from site $\,x\,$ to site $\,y\,$ given by $\,q^n_{y,x}\,$. We note further that $\,\widehat{q}^{\,\,n}\,$ is given by the sum of the generating matrix of this \textsc{Markov} chain $\,\hat{Y}^n(\cdot)\,$, and of a diagonal matrix with diagonal entries given by the values of the function $\,V^n\,$ from \eqref{potential}. 

It   follows now from the \textsc{Feynman-Kac} formula  that $\,(\hat{P}^n_t)_{t\ge0}\,$ is the \textsc{Feynman-Kac} semigroup corresponding to this \textsc{Markov} chain $ \hat{Y}^n(\cdot) $ and potential $\,V^n\,$; see  for instance \cite{RoW}, section IV.22, Example (22.11). In other words, for each $\,t\in[0,\infty)\,$ and $\,x\in\,\mathcal{Y}^n\,$, the measure $\,\delta_x\hat{P}^n_t\,$ is absolutely continuous with respect to the distribution at time $\,t\,$ of the   \textsc{Markov} chain $\,\hat{Y}^n(\cdot)\,$ started from $\,x\,$,  with   density given by 
\[
\exp\Big(\int_0^t V^n \big(\hat{Y}^n(s)\big)\,\mathrm{d}s\Big).
\] 

Note that the \textsc{Markov} chain $\,\hat{Y}^n(\cdot)\,$ does not admit a decomposition of the form \eqref{dual_decomp}, because the compensator of its jumps originating from the boundary of $\,\mathcal{Y}^n\,$ is not of the form $\,\mathfrak{R}^*\,\hat{L}^n(\cdot)\,$. However, one can change the jump rates   of $\,\hat{Y}^n(\cdot)\,$ from and to $\,\partial\mathcal{Y}^n\,$, to rates of the type \eqref{rate1}-\eqref{rate4}, \eqref{boundaryrates1}, \eqref{boundaryrates} (with $\,\mathfrak{b}\,$ replaced by $\,0\,$ and $\,\mathfrak{R}\,$ replaced by $\,\mathfrak{R}^*\,$) by means of an equivalent change of measure, so that the decomposition \eqref{dual_decomp} starts to hold under the new measure. Denoting the resulting \textsc{Markov} chain by $\,\tilde{Y}^n(\cdot)\,$ and its generating matrix by $\,\tilde{q}^{\,n}\,$, and changing the measure back to the original one, we end up with \eqref{com}.       

 \smallskip
Now we turn to the proof of \eqref{duality0}. To this end, we let $\,v_f$, $v_g\,$ be vectors with coordinates indexed by the elements of $\,\mathcal{Y}^n\,$ such that the $\,x$-th coordinate of $\,v_f\,$ is given by $\,f(x)\,$ and the $\,x$-th coordinate of $\,v_g\,$ is given by $\,g(x)\,$ for each $\,x\,$. Then \eqref{duality0} can be cast equivalently as
\eq\label{vector_eq}
(v_f)'\,\exp \big(\,\widehat{q}^{\,\,n}\,t\big)\,v_g=(v_g)'\,\exp \big(\,q^{\,n}\,t\big)\,v_f\,.
\en
Clearly now, \eqref{vector_eq} is a consequence of $\,\widehat{q}^{\,\,n}=(q^n)'\,$. 
\end{proof}

From Lemma \ref{duality_disc} we can deduce the following duality relation for the RBMs $\,X(\cdot)$, $\tilde{X} (\cdot)\,$ of Theorem \ref{main_thm}, which is of interest in its own right. 

\begin{thm} \label{dual}
Let $\,(P_t)_{t\ge0}\,$ be the transition semigroup of the RBM $\,X(\cdot)\,$ in Theorem \ref{main_thm}, and define $\,\big(\widehat{P}_t\big)_{t\ge0}\,$ as the \textsc{Feynman-Kac} semigroup on $\,\rr_+^d\,$ given by
\eq
(\hat{P}_t\,f)(x)=\ev^x\Big[\exp\Big(-2\,\sum_{i=1}^d \,\frac{\mathfrak{b}_i\,\mathfrak{r}_{i,i}}{\mathfrak{a}_{i,i}}\,\tilde{L}_i(t)\Big)\,f \big(\tilde{X}(t)\big)\Big],\qquad f\in C_c(\rr_+^d).
\en
Then the two semigroups are in duality with respect to   \textsc{Lebesgue}  measure on $\,\rr_+^d\,$:
\eq\label{duality}
\forall\,f,\,g\in C_c(\rr_+^d),\;t\ge0:\quad \int_{\rr_+^d} f(x)\,(\hat{P}_t\,g)(x)\,\mathrm{d}x
= \int_{\rr_+^d} (P_t\,f)(x)\,g(x)\,\mathrm{d}x.
\en 
\end{thm}

\begin{proof} The   idea of the argument is  to obtain \eqref{duality} by   plugging   \eqref{com} into \eqref{duality0} and   taking the limit as $\,n\to\infty\,$. To this end, we study the asymptotics of the expression in \eqref{com} as $\,n\to\infty\,$ in Step 1 of the proof, and pass to the limit $\,n\to\infty\,$ in \eqref{duality0} in Steps 2 and 3. 

\medskip
\noindent\textit{Step 1.} We note first that the identity $\,\hat{q}^{\,n}_{x,y}= \tilde{q}^{\,n}_{x,y}=q^n_{y,x}\,$   for all $\,x\notin\partial_1\mathcal{Y}^n\,$ allows to rewrite \eqref{com} as
\begin{eqnarray*}
\exp\Big(\sum_{x\in\mathcal{Y}^n} \tilde{T}^n_x(t)\,\sum_{y\sim x} (q^n_{y,x}-q^n_{x,y})\Big)
\prod_{x\in\partial_1\mathcal{Y}^n} \prod_{y\sim x} \Big(\frac{q^n_{y,x}}{\tilde{q}^n_{x,y}}\Big)^{\tilde{N}^n_{x,y}(t)}
\,\exp\Big(-\sum_{x\in\mathcal{Y}^n} \tilde{T}^n_x(t)\,\sum_{y\sim x} (q^n_{y,x}-\tilde{q}^n_{x,y})\Big) \\
=\exp\Big(\sum_{x\in\mathcal{Y}^n} \tilde{T}^n_x(t)\,\sum_{y\sim x} (\tilde{q}^n_{x,y}-q^n_{x,y})\Big)
\prod_{x\in\partial_1\mathcal{Y}^n} \prod_{y\sim x} \Big(\frac{q^n_{y,x}}{\tilde{q}^n_{x,y}}\Big)^{\tilde{N}^n_{x,y}(t)}.
\end{eqnarray*}
Moreover, we choose the jump rates of $\,\tilde{Y}^n(\cdot)\,$ in such a way that $\,\sum_{y\sim x} (\tilde{q}^n_{x,y}-q^n_{x,y})=0\,$ for all $\,x\in\mathcal{Y}^n\,$, which reduces the above expression to 
\eq\label{simplified}
\prod_{x\in\partial_1\mathcal{Y}^n} \prod_{y\sim x} \Big(\frac{q^n_{y,x}}{\tilde{q}^n_{x,y}}\Big)^{\tilde{N}^n_{x,y}(t)}.
\en
More specifically, for every fixed $\,i\in\{1,\ldots,d\}\,$, we define the jump rates $\,z_j\,$ of $\,\tilde{Y}^n(\cdot)\,$ from sites in $\,\mathcal{B}_{\{i\}}\,$ in the directions $\,e_i/\sqrt{n}+\mathrm{sgn}(\mathfrak{a}_{i,j})\,e_j/\sqrt{n}$, $j\in\{1,\ldots,d\}\backslash\{i\}\,$, and the jump rate $\,z_i\,$ in the direction $\,e_i/\sqrt{n}\,$, by solving the system of equations
\eq
\frac{z_j}{n\,|\mathfrak{a}_{i,j}|/2}=\frac{z_i}{n\,\mathfrak{a}_{i,i}/\zeta_i-\sqrt{n}\,\mathfrak{b}_i}\,,
\quad j\in\{1,\ldots,d\}\backslash\{i\}\,,
\qquad\qquad\sum_{k=1}^d z_k=n\,\mathfrak{r}_{i,i}\,,
\en 
where we set $\,z_j=0\,$ and eliminate $\,z_j\,$ from the system of equations whenever $\,\mathfrak{a}_{i,j}=0\,$. In particular, we obtain
\eq
z_i=n\,\mathfrak{r}_{i,i}\,
\frac{n\,\mathfrak{a}_{i,i}/\zeta_i-\sqrt{n}\,\mathfrak{b}_i}{n\,\mathfrak{a}_{i,i}/2-\sqrt{n}\,\mathfrak{b}_i}\,.
\en

\medskip
\noindent
In addition, we select the constants in \eqref{boundaryrates1}, \eqref{boundaryrates} and in the definition of the corresponding jump rates of $\,\tilde{Y}^n(\cdot)\,$ so that, for every $\,I\subset\{1,\ldots,d\}\,$, the factors in \eqref{simplified} with $\,x,y\in\mathcal{B}_I\,$ are equal to $\,1\,$; and, for every pair of distinct sets $\,I\subset J\subset\{1,\ldots,d\}$ with $I\neq J$, $|J|\geq2\,$, the ratios $\,\frac{q^n_{y,x}}{\tilde{q}^n_{x,y}}$, $\frac{q^n_{x,y}}{\tilde{q}^n_{y,x}}\,$ entering \eqref{simplified} with $\,x\in\mathcal{B}_I$, $y\in\mathcal{B}_J\,$ depend only on $\,|I|,|J|\,$, with    
\eq
\gamma_{|I|,|J|}\,:=\,\frac{q^n_{y,x}}{\tilde{q}^{\,n}_{x,y}}\,=\,\Big(\frac{q^{\,n}_{x,y}}{\tilde{q}^n_{y,x}}\Big)^{-1}\quad\mathrm{and}\quad\gamma_{0,2}=\gamma_{1,2}\,. 
\en
Such a choice of constants can be found by starting with $\,\mathcal{B}_{\{1,\ldots,d\}}\,$, proceeding successively to boundaries of higher dimensions, and increasing the constants chosen before in each step if necessary. With these definitions, and in the light of \eqref{rate2} and \eqref{boundaryrates}, the product in \eqref{simplified} simplifies to  
\eq\label{simplified2}
\begin{split}
\gamma_{d-1,d}^{\tilde{M}^n_{d-1,d}(t)-\tilde{M}^n_{d,d-1}(t)}\cdot
\prod_{j=3}^{d-1} \Big(\gamma_{j-1,j}^{\tilde{M}^n_{j-1,j}(t)-\tilde{M}^n_{j,j-1}(t)}\,\gamma_{j,j+1}^{\tilde{M}^n_{j+1,j}(t)-\tilde{M}^n_{j,j+1}(t)}\Big)
\cdot\gamma_{1,2}^{\tilde{M}^n_{1,2}(t)+\tilde{M}^n_{0,2}(t)-\tilde{M}^n_{2,1}(t)-\tilde{M}^n_{2.0}(t)}    \\
\cdot\prod_{i=1}^d \,\Big( \frac{\,2\,n\,\mathfrak{r}_{i,i}/ \zeta_i\,} 
{n\,\mathfrak{a}_{i,i}/\zeta_i}\Big)^{\tilde{N}^n_{\emptyset,\{i\}}(t)}
\,\Big(\frac{n\,\mathfrak{a}_{i,i}/2-\sqrt{n}\,\mathfrak{b}_i}{n\,\mathfrak{r}_{i,i}}\Big)^{\tilde{N}^n_{\{i\},\emptyset}(t)}\,,
\end{split}
\en 
where we have written $\,\tilde{M}^n_{i,j}(t)\,$ for the total number of jumps of $\,\widetilde{Y}^n (\cdot)\,$ from sites belonging to a boundary $\,\mathcal{B}_I\,$ with $\,|I|=i\,$ to sites belonging to a boundary $\,\mathcal{B}_J\,$ with $\,|J|=j\,$ during the time interval $\,[0,t]\,$. 

\smallskip

Now, for $\,\tilde{Y}^n(0)\rightarrow x\in\mathcal{B}_\emptyset\,$, the convergence in distribution $\,\tilde{Y}^n(\cdot)\Longrightarrow\tilde{X}(\cdot)\,$ (due to an analogue of Theorem \ref{conv} for $\,\tilde{X}(\cdot)\,$) and the fact that $\,\tilde{X}(t)\,$ takes values in $\,\mathcal{B}_\emptyset\,$ with probability $1$ imply that the probability of the event $\,\{\tilde{Y}^n(0),\tilde{Y}^n(t)\in\mathcal{B}_\emptyset\}\,$ tends to $1$. Moreover, $\,\tilde{Y}^n(\cdot)\,$ can reach $\,\mathcal{B}_J\,$ from $\,\mathcal{B}_I\,$ for a set $\,J\neq I\,$ in one jump, if and only if $\,\big||I|-|J|\big|=1$ or $\,|I|=2,|J|=0$ or $\,|I|=0,|J|=2\,$. Putting the latter two facts together, we conclude that the exponents in the first line of \eqref{simplified2} converge to $0$ in distribution. In addition, we can put the same argument together with the convergences 
\[
\Big(t\mapsto\sqrt{n}\cdot\int_0^\cdot \mathbf{1}_{\{\tilde{Y}^n_i(s)=0\}}\,\mathrm{d}s\Big) 
\Longrightarrow \tilde{L}_i(\cdot)\quad\mathrm{and}\quad
\Big(t\mapsto\sqrt{n}\cdot\int_0^\cdot \mathbf{1}_{\{\tilde{Y}^n_i(s)=\tilde{Y}^n_j(s)=0\}}\,\mathrm{d}s\Big) 
\Longrightarrow 0
\]
to conclude that the exponents in the second line of \eqref{simplified2} are both on the order of $\,\sqrt{n}\,$ with the corresponding prefactor being given by $\,\mathfrak{r}_{i,i}\,\tilde{L}^n_i(t)\,$. The latter can be computed by viewing $\,\tilde{N}_{\{i\},\emptyset}(t)\,$ as the value of a standard \textsc{Poisson} process running under the clock $\,n\,\mathfrak{r}_{i,i}\,\int_0^t \mathbf{1}_{\{\tilde{Y}^n_i(s)=0\}}\,\mathrm{d}s\,$ and applying the Law of Large Numbers for \textsc{Poisson} processes. 

All in all, we see that the product in \eqref{simplified2} is asymptotically equivalent to 
\[
\prod_{i=1}^d \Big(\frac{2\,n\,\mathfrak{r}_{i,i}/\zeta_i}{n\,\mathfrak{a}_{i,i}/\zeta_i}
\cdot\frac{n\,\mathfrak{a}_{i,i}/2-\sqrt{n}\,\mathfrak{b}_i}{n\,\mathfrak{r}_{i,i}}\Big)
^{\sqrt{n}\,\mathfrak{r}_{i,i}\,\tilde{L}^n_i(t)}
=\prod_{i=1}^d \Big(1-\frac{2\,\mathfrak{b}_i}{\sqrt{n}\,\mathfrak{a}_{i,i}}\Big)^{\sqrt{n}\,\mathfrak{r}_{i,i}\,\tilde{L}^n_i(t)} \\
\Longrightarrow \prod_{i=1}^d e^{-2\,\frac{\mathfrak{b}_i\,\mathfrak{r}_{i,i}}{\mathfrak{a}_{i,i}}\,\tilde{L}_i(t)}.  
\]
Here, by {\it asymptotic equivalence} of two processes, we mean that the first of the two processes converges in distribution if and only if the second one converges in distribution, and in this case the two limits are indistinguishable. 

\medskip

\noindent\textit{Step 2.} We claim now that 
\eq\label{conv_inside}
\lim_{n\to\infty} (\hat{P}^n_t\,g)(x) 
= \ev^x\Big[\exp\Big(-2\,\sum_{i=1}^d \frac{\mathfrak{b}_i\,\mathfrak{r}_{i,i}}
{\mathfrak{a}_{i,i}}\,\tilde{L}_i(t)\Big)\,g(\tilde{X}(t))\Big]
\en
for all $\,x\in\rr_{++}^d\,$ and $\,t\ge0\,$ with the notation of \eqref{duality0}. In view of \eqref{com}, Step 1 and Theorem \ref{conv}, it suffices to show that for a $\,p>1\,$ the $\,p$-th moments of the random variables in \eqref{simplified2} are bounded uniformly in $\,n\,$. Indeed,  these random variables multiplied by $\,g(\tilde{Y}^n(t))\,$ would converge then to the random variable inside the expectation in \eqref{conv_inside} in $\,\mathbb{L}^1$ by the \textsc{Vitali}  Convergence Theorem. 

To prove the uniform boundedness of the moments, we observe from the explicit formula for the moment generating function of a \textsc{Poisson} random variable that it is enough to show
\eq\label{unifint}
\mathfrak{s}_p(x):=\sup_{n\in\nn}\;\ev^x\bigg[\,\prod_{i=1}^d \Big(1-\frac{2\,\mathfrak{b}_i}{\,\sqrt{n\,}\,\mathfrak{a}_{i,i}\,}\Big)^{(e^p-1)\,\sqrt{n}\,\mathfrak{r}_{i,i}\,\tilde{L}^n_i(t)}\,\bigg]<\infty
\en     
for a fixed $\,p>1\,$. Due to the decomposition \eqref{dual_decomp}, we can apply the oscillation inequality of Theorem 5.1 in \cite{Wi4} to find a constant $\,C_p<\infty\,$ depending only on $\,\mathfrak{R}^*\,$ and $\,p\,$, such that
\eq\label{oscbound}
\mathfrak{s}_p(x)\leq \sup_{n\in\nn}\;\ev^x\bigg[\prod_{i=1}^d \Big(1-\frac{2\,\mathfrak{b}_i}{\sqrt{n}\,\mathfrak{a}_{i,i}}\Big)^{C_p\,\sqrt{n}\,\sup_{0\le u\le t} \left|\tilde{C}^n(u)+\tilde{M}^n(u)\right|}\,\bigg].
\en
In addition, $\,\sup_{0\le u\le t} |\tilde{C}^n(u)|\,$ can be bounded uniformly in $\,n\,$, and $\,\tilde{M}^n (\cdot)\,$ can be viewed as a \textsc{Lipschitz} time change (with the \textsc{Lipschitz} constant bounded uniformly in $\,n\,$) of a compensated \textsc{Markov} chain $\,\bar{S}^n(\cdot)\,$ converging to $\,B(\cdot)\,$ and with jump rates being homogeneous in space and time. Hence, the right-hand side of \eqref{oscbound} is certainly finite if, for any fixed $\,t\,$, any fixed exponential moment of $\,\sup_{0\le u\le t} |\bar{S}^n(u)|\,$ can be bounded uniformly in $\,n\,$. The latter is a routine exercise that we leave to the reader. We also note that the resulting bound on $\,\mathfrak{s}_p(x)\,$ is independent of $\,x\,$.

\medskip
\noindent\textit{Step 3.} 
To complete the proof, we define for each $\,n\in\nn\,$ the function $\,\kappa_n:\;\rr_+^d\rightarrow\frac{1}{\sqrt{n}}\,\zz_+^d\,$ which rounds down all coordinates of a vector $\,x\in\rr_+^d\,$ to the nearest element of the lattice $\frac{1}{\sqrt{n}}\,\zz_+^d\,$. With this notation, and for $\,n\,$ large enough (more specifically, such that $\,[0,K_n]^d\,$ contains the supports of $\,f\,$ and $\,g\,$ of \eqref{duality}), the duality identity \eqref{duality0} can be cast as
\eq\label{prelim}
\forall\,\;t\ge0:\quad \int_{\rr_+^d} f(\kappa_n(x))\,\big(\hat{P}^n_t\,g\big)(\kappa_n(x))\,\mathrm{d}x
= \int_{\rr_+^d} \ev^{\kappa_n(x)}[f(Y^n(t))]\,g(\kappa_n(x))\,\mathrm{d}x.
\en  
Choosing cubes $R^f$, $R^g$ whose vertices have nonnegative integer coordinates and such that $R^f$ contains the support of $f$ and $R^g$ contains the support of $g$, we see that the integrands in \eqref{prelim} are bounded above by $\,\mathbf{1}_{R^f}\,\|f\|_\infty\,\|g\|_\infty\,\mathfrak{s}_1(\kappa_n(x))$, $\,\mathbf{1}_{R^g}\,\|f\|_\infty\,\|g\|_\infty\,$, respectively. Now, passing to the limit $n\to\infty$ in \eqref{prelim} and using \eqref{conv_inside} and the Dominated Convergence Theorem we end up with \eqref{duality}.    
\end{proof}

\section{Proof of the main result} \label{sec4}

At this stage, Theorem \ref{main_thm} is a rather simple consequence of Theorem \ref{dual}.

\medskip

\noindent\textit{Proof of Theorem \ref{main_thm}.} Since the filtration $\mathbb{F}$ is generated by projections, the first statement in Theorem \ref{main_thm} will follow if we can show
$$
 \ev^{\mathbb{Q}}\Big[\prod_{j=1}^\ell f_j(\widetilde{X}( t_j))\Big] \,=\,\ev^{\pp}\Big[\prod_{j=1}^\ell f_j(X(T-t_j))\Big]
$$
or equivalently 
\eq\label{mainiden}
\ev^{\tilde{\pp}}\bigg[\exp\Big(-2\,\sum_{i=1}^d \frac{\mathfrak{b}_i\,\mathfrak{r}_{i,i}}
{\mathfrak{a}_{i,i}}\,\tilde{L}_i(T)\Big)
\,\frac{p(\tilde{X}(T))}{p(\tilde{X}(0))}\,\prod_{j=0}^\ell f_j(\tilde{X}(t_j))\bigg]
\,=\,\ev^{\pp}\Big[\prod_{j=0}^\ell f_j(X(T-t_j))\Big]
\en
for all $\,\ell\in\nn$, $0=t_0<t_1<\ldots<t_\ell=T\,$ and $\,f_0,f_1,\ldots,f_\ell\in C_c(\rr_+^d)\,$. In particular, let us note that \eqref{mainiden} implies that $\,\mathbb{Q}\,$ is a probability measure: indeed, one only needs to take $\,\ell=1\,$ and sequences of nonnegative functions $\,f_0$, $f_1\,$ which increase to the function $\, f \equiv 1\,$, and to apply the Monotone Convergence Theorem.   

\medskip

To show \eqref{mainiden} we apply \eqref{duality} repeatedly to the left-hand side of \eqref{mainiden}:
\begin{eqnarray*}
&&\int_{\rr_+^d} \mathrm{d}x\,f_0(x) \int_{\rr_+^d} \hat{P}_{t_1}(x,\mathrm{d}x_1)\,f_1(x_1)
\ldots\int_{\rr_+^d} \hat{P}_{t_\ell-t_{\ell-1}}(x_{\ell-1},\mathrm{d}x_\ell)\,f_\ell(x_\ell)\,p(x_\ell) \\
&&=\int_{\rr_+^d} \mathrm{d}x\,f_1(x)\,\ev^x[f_0(X(t_1))] \int_{\rr_+^d} \hat{P}_{t_2-t_1}(x,\mathrm{d}x_2)\,f_2(x_2)
\ldots\int_{\rr_+^d} \hat{P}_{t_\ell-t_{\ell-1}}(x_{\ell-1},\mathrm{d}x_\ell)\,f_l(x_\ell)\,p(x_\ell) \\
&&\vdots \\
&&=\int_{\rr_+^d} \mathrm{d}x\,\ev^x[f_0(X(t_\ell))\,f_1(X(t_\ell-t_1))\ldots f_{\ell-1}(X(t_\ell-t_{\ell-1}))]\,f_\ell(x)\,p(x).   
\end{eqnarray*}
Now, the stationarity of $\,X\,$ under $\,{\bm\nu}\,$ shows that the latter expression coincides with the right-hand side of \eqref{mainiden}. 

\medskip

If the probability density function $\,p(\cdot)\,$ is twice continuously differentiable and strictly positive throughout $\,\rr_+^d\,$, then we can apply  \textsc{It\^o}'s formula to $\,\log p(\tilde{X}(t))\,$ and obtain 
\eq\label{Ito}
\begin{split}
M(t):=\,\exp\Big(-2\sum_{j=1}^d \frac{\mathfrak{b}_j\,\mathfrak{r}_{j,j}}{\mathfrak{a}_{j,j}}\tilde{L}_j(t)\Big)
\,\frac{p(\tilde{X}(t))}{p(\tilde{X}(0))} \,=\, Z(t)\cdot \exp\bigg(\int_0^t \bigg(\frac{\tilde{\mathcal{A}}\,p}{p}\bigg)\big(\tilde{X}(s)\big)\,\mathrm{d}s\bigg)
\\
\cdot \, \exp\Big(-2\sum_{j=1}^d \frac{\mathfrak{b}_j\,\mathfrak{r}_{j,j}}{\mathfrak{a}_{j,j}}\tilde{L}_j(t) 
+ \sum_{j=1}^d \int_0^t \nabla\log p(\tilde{X}(s))\cdot \overline{\mathfrak{r}}^{\,*}_j\,\mathrm{d}\tilde{L}_j(s)\Big).  
\end{split}
\en
Here $\,\tilde{\mathcal{A}}\,$ is the generator of the process $\,B(t) -\mathfrak{b}\,t\,$, $t\ge0\,$, the process $\, Z(\cdot)\,$ is the exponential $\,\widetilde{\mathbb{P}}-$local martingale of (\ref{expmart}), and $\, \overline{\mathfrak{r}}^{\,*}_j= ( \mathfrak{r}_{1,j}^*, \cdots, \mathfrak{r}_{d,j}^*)^\prime$ the $j$-th column of the dual reflection matrix $\,\mathfrak{R}^*= (\mathfrak{r}_{i,j}^*)_{1 \le i, j \le d}\,$.  

\smallskip

Next, we write $\,\mathfrak{A}=U\,\mathfrak{L}\,U^{-1}\,$ where $\,U\,$ is an orthogonal matrix whose columns are orthonormal eigenvectors of $\,\mathfrak{A}\,$ and $\,\mathfrak{L}\,$ is the corresponding diagonal matrix of eigenvalues of $\,\mathfrak{A}\,$. Clearly, $\,x \mapsto p\left(U^{-1} \mathfrak{L}^{1/2} x\right)\,$ is an invariant density of the RBM $\,\mathfrak{L}^{-1/2}U\tilde{X}(\cdot)\,$ with unit diffusion matrix. At this point, the computation of the normal and tangential components of the reflection matrix of $\,\mathfrak{L}^{-1/2}U\tilde{X}(\cdot)\,$ in section 3.2.1 of \cite{IK} together with the proof of Lemma 7.1 in \cite{HW2} show that, as a consequence of the basic adjoint relationship satisfied by $\,x\mapsto p\left(U^{-1} \mathfrak{L}^{1/2} x\right)\,$, one has  
\[
-2\,\frac{\mathfrak{b}_j\,\mathfrak{r}_{j,j}}{\mathfrak{a}_{j,j}}\,p(x)+(\nabla p)(x)\cdot\overline{\mathfrak{r}}_j^*=0 \quad\mathrm{for\;all\;}~x\in\rr_+^d\;~\mathrm{with\;}~x_j=0
\]
and
\[
(\tilde{\mathcal{A}}\,p)(x)=0\quad\mathrm{for\;all\;}~x\in\rr_{++}^d\,.   
\]

\medskip
\noindent

Therefore, $\,M(\cdot)\,$ is equal to the exponential $\,\widetilde{\mathbb{P}}-$local martingale $\,Z(\cdot)\,$ of \eqref{expmart}. We conclude from this reasoning that $\, M(\cdot)\,$ is a positive $\,\widetilde{\mathbb{P}}-$supermartingale; but we have argued already that
$$
\mathbb{E}^{\widetilde{\mathbb{P}}}\, \big[ M (T) \big] \,=\, \mathbb{Q}(\Omega) \,=\, 1 \,=\, M(0)\,,
$$
so in fact $\, M(\cdot)\,$ is a  $\,\widetilde{\mathbb{P}}-$martingale. Because $\, M(\cdot) \equiv Z(\cdot)\,$,  
the decomposition of \eqref{semimart} is now a consequence of   \eqref{expmart} and of the \textsc{Girsanov}  Theorem. \ep

\bigskip\bigskip

 
\bibliographystyle{alpha}

\begin{thebibliography}{50}

\bibitem{BW} \textsc{Bhardwaj, S.} \& \textsc{Williams, R.J.} (2009) Diffusion approximation for a heavily loaded multi-user wireless communication system with cooperation. \textit{Queueing Syst.} \textbf{62}, 345--382.

\bibitem{Bi} \textsc{Billingsley, P.} (1999) \textit{Convergence of Probability Measures}. Second Edition. Wiley Series in Probability and Statistics. J. Wiley \& Sons, New York.


\bibitem{EK} \textsc{Ethier, S.N.} \& \textsc{Kurtz, T.G.} (1986) \textit{Markov Processes: Characterization and Convergence}. Wiley Series in Probability and Statistics. J. Wiley \& Sons, New York.  

\bibitem{HR} \textsc{Harrison, J.M.} \& \textsc{Reiman, M.I.} (1981) Reflected Brownian motion in an orthant. \textit{Ann. Probab.} \textbf{9},  302--308.

\bibitem{HW} \textsc{Harrison, J.M.} \& \textsc{Williams, R.J.} (1987) Brownian models of open queueing networks with homogeneous customer populations. \textit{Stochastics} \textbf{22}, 77--115.

\bibitem{HW2} \textsc{Harrison, J.M.} \& \textsc{Williams, R.J.} (1987) Multidimensional reflected Brownian motions having exponential stationary distributions. \textit{Ann. Probab.} \textbf{15}, 115-137.

\bibitem{HP} \textsc{Haussmann, U. \& Pardoux, E.} (1986) Time reversal of diffusions. {\it Ann. Probab.} {\bf 14}, 1188-1205.

\bibitem{IK} \textsc{Ichiba, T.} \& \textsc{Karatzas, I.} (2010) On collisions of Brownian particles. {\it Ann. Appl. Probab.} \textbf{20}, 951--977.

\bibitem{KPS} \textsc{Karatzas, I.}, \textsc{Pal, S.} \& \textsc{Shkolnikov, M.} (2012) Systems of Brownian particles with asymmetric collisions. Preprint. Available at \textit{arxiv.org/abs/1210.0259}. 

\bibitem{M} \textsc{Meyer, P.A.}  (1994)  Sur une transformation du mouvement brownien due  \`a Jeulin et Yor. {\it Lecture Notes in Mathematics}   {\bf  1583}, 98-101. Springer-Verlag, NY.

\bibitem{R} \textsc{Reiman, M.I.} (1984) Open queueing networks in heavy traffic. \textit{Mathematics of Operations Research} \textbf{9}, 441--458.

\bibitem{RW} \textsc{Reiman, M.I.} \& \textsc{Williams, R.J.} (1988) A boundary property of semimartingale reflecting Brownian motions. \textit{Probab. Theory Relat. Fields} \textbf{77}, 87--97.

\bibitem{RoW}  \textsc{Rogers, L.C.G. \& Williams, D.} (1987) {\it Diffusions, Markov Processes and Martivgales. Volume 2: It\^o Calculus.} Wiley Series in Probability and Statistics. J. Wiley \& Sons, New York. 

\bibitem{TW} \textsc{Taylor, L.M.} \& \textsc{Williams, R.J.} (1993) Existence and uniqueness of semimartingale reflecting Brownian motions in an orthant. \textit{Probab. Theory Relat. Fields} \textbf{96}, 283--317.

\bibitem{Wi3} \textsc{Williams, R.J.} (1987) Reflected Brownian motion with skew symmetric data in a polyhedral domain. \textit{Probab. Theory Rel. Fields} \textbf{75}, 459--485.

\bibitem{Wi} \textsc{Williams, R.J.} (1988) On time-reversal of reflected Brownian motions. \textit{Seminar on Stochastic Processes 1987,} 265--276. {\sl Progress in Probability and  Statistics} \textbf{15}, Birkh\"auser,  Boston.

\bibitem{Wi2} \textsc{Williams, R.J.} (1995) Semimartingale reflecting Brownian motions in the orthant. In \textit{Stochastic Networks} (F.P. Kelly and R.J. Williams, editors). Springer Verlag, New York. 

\bibitem{Wi4} \textsc{Williams, R.J.} (1998) An invariance principle for semimartingale reflecting Brownian motions in an orthant. \textit{Queueing Syst.} \textbf{30}, 5--25.
\end{thebibliography}

\bigskip

\end{document}